\documentclass[reqno]{amsart}

\usepackage{amsmath}
\usepackage{amssymb}
\usepackage{amsfonts}
\usepackage{graphicx}
\usepackage{amsthm}
\usepackage{enumerate}
\usepackage[mathscr]{eucal}
\usepackage{lscape}
\usepackage{dsfont}
\usepackage{color}
\usepackage{mathtools}

\usepackage{setspace}
\newtheorem{theor}{Theorem}[section]

\newtheorem{corol}[theor]{Corollary}
\newtheorem{lem}[theor]{Lemma}

\newcommand{\tyz}{Tian-Yau-Zelditch }

\newcommand{\Z}{\mathbb{Z}}
\newcommand{\CP}{\mathds{C}\mathrm{P}}
\newcommand{\de}{\partial}

\newcommand{\R}{\mathbb{R}}
\newcommand{\B}{\mathcal{B}}

\newcommand{\C}{\mathbb{C}}

\newcommand{\N}{\mathbb{N}}

\newcommand{\K}{K\"{a}hler\ }

\renewcommand{\div}{\textrm{div}}
\newcommand{\Ric}{\textrm{Ric}}

\begin{document}
\title[On the third coefficient of TYZ expansion for radial metrics]{On the third coefficient of TYZ expansion for radial scalar flat metrics}

\author{Andrea Loi, Filippo Salis, Fabio Zuddas}
\address{Dipartimento di Matematica e Informatica, Universit\`a di Cagliari\\Via Ospedale 72, 09124 Cagliari (Italy)}
\email{loi@unica.it, filippo.salis@gmail.com, fabio.zuddas@unica.it}

\thanks{The first author was  supported by Prin 2015 -- Real and Complex Manifolds; Geometry, Topology and Harmonic Analysis -- Italy  by INdAM. GNSAGA - Gruppo Nazionale per le Strutture Algebriche, Geometriche e le loro Applicazioni and by Fondazione di Sardegna and Regione Autonoma della Sardegna.}

\begin{abstract}
We classify radial scalar flat metrics with constant  third coeffcient of its TYZ expansion.  
As a byproduct of our analysis we provide a characterization of  Simanca's scalar flat metric. 
\end{abstract}
 
\maketitle
\tableofcontents

\section{Introduction}
A K\"ahler metric $g$ on a complex manifold $M$ is said to be \emph{projectively induced} if there exists a K\"ahler (isometric and holomorphic) immersion of $(M, g)$ into the complex projective space $(\C P^N, g_{FS})$, $N \leq +\infty$,  endowed with the Fubini--Study metric $g_{FS}$, namely the metric whose associated \K form is given in homogeneous coordinates  by  $\omega_{FS}=\frac{i}{2}\partial\bar\partial\log (|Z_0|^2+\cdots +|Z_N|^2)$. Since requirement that a \K metric is projectively induced is a somehow strong assumption, one could try to approximate an integral  \K\ form
on a  complex manifold with suitable normalized  projectively induced \K\ forms through the following construction.

Assume that there exists a Hermitian line bundle $(L,h)$ over  $M$ whose Ricci curvature $\Ric (h)$ equals $\omega$ (this is always possible in the compact case).
For every postive integer $m$  one considers the holomorphic line bundle $L^m=L^{\otimes m}$ endowed with the Hermitian metric $h_m$ induced on $L^m$ by $h$, such
that $\Ric (h_m)=m\omega$.
Let $\mathscr{H}_m$ be the separable complex Hilbert space consisting of global holomorphic sections  of $L^m$ such that
$$\langle s,s \rangle_m=\int_M h_m(s,s)\frac{\omega^n}{n!}<\infty.$$
If $s_j$, $j=0,\mathellipsis , N_{m}$,  $N_{m}+1=\dim\mathscr{H}_m\leq\infty$ is an orthonormal basis of $\mathscr{H}_m$, the smooth function
$$\epsilon_{mg}(x) =\sum_{j=0}^{N_m} h_m(s_j(x),s_j(x))$$
is globally defined on $M$ and, as suggest by the notation, it does not depend on the orthonormal basis or the Hermitian metric chosen (see e.g. \cite{cgr1}).
This function, also known in literature as \emph{distortion function}, constitutes a tool to evaluate how a \K metric differs from being projectively induced. Indeed, if there exists a sufficiently large integer $m$ such that, for every point of $M$, we can find at least one element of an orthonormal basis $\{s_j\}_{j=0,\mathellipsis ,N_m}$ of $\mathscr{H}_m$ which does not vanishes at this point (the free based point condition in the compact case),  \emph{coherent state map}, namely the holomorphic map 
\begin{align*} 
\varphi_m:  M & \to \CP^{N_m} \\
x &\mapsto [s_0(x), \mathellipsis, s_{N_m}(x) ],
\end{align*}
is well-defined and, moreover (see, e.g. \cite{arezzoloi} for a proof), it satisfies
$$\varphi_m^*\omega_{FS}= m\omega+\frac{i}{2}\partial\bar\partial\log\epsilon_{m g}.$$
Although not all \K metrics are projectively induced, Tian (\cite{tian4}) and Ruan (\cite{ruan}) solved  a conjectured by Yau by proving that  any polarized metric on a compact complex manifold is the $C^\infty$-limit of normalized projectively induced metrics $\frac{\varphi_m^*g_{FS}}{m}$. Then, Catlin (\cite{catlin}) and Zelditch (\cite{zelditch}) independently generalized the Tian-Ruan theorem by proving the existence of a complete asymptotic expansion for the distortion function related to any polarized metric defined on compact complex manifold.
This  asymptotic expansion  
\begin{equation}\label{TYZ}
\epsilon_{m g}(x)\sim \sum_{j=0}^\infty a_j(x)m^{n-j},
\end{equation}
is called {\em Tian--Yau--Zelditch expansion} and it means that there exists a positive constant $C_{l,r}$ depending on  two positive integer constants $l$ and $r$ such that
$$\Big\Vert\epsilon_{m g}(x) -\sum_{j=0}^l a_j(x)m^{n-j}\Big\Vert_{C^r}\leq\frac{C_{l,r}}{m^{l+1}},$$
where $a_0(x)=1$ and $a_j(x)$, $j=1,\dots$ are smooth functions on $M$.
In particular, Z. Lu \cite{lu}  computed the first three TYZ coefficients.
The expression of the first two coefficients  are (for the rather involved expression of third coefficient $a_3$ see (\ref{coeffa3}) below):
 \begin{equation}\label{coefflu}
\left\{\begin{array}
{l}
a_1(x)=\frac{1}{2}\rho\\
a_2(x)=\frac{1}{3}\Delta\rho
+\frac{1}{24}(|R|^2-4|{\rm Ric} |^2+3\rho ^2)\\
\end{array}\right.
\end{equation}
where 
$\rho$, $R$, $\Ric$ denote respectively the scalar curvature,
the curvature tensor and the Ricci tensor of $(M, g)$,

Due to Donaldson's work (cf. \cite{donaldson}, \cite{do2} and \cite{arezzoloi}) in the compact case and respectively to the theory of quantization  in the noncompact case (see e.g. \cite{Ber1}, \cite{cgr3} and \cite{cgr4}), it is  natural to study metrics with the \tyz coefficients being prescribed. For instance, the vanishing of this coefficients for large enough indexes  turns out to be related to some important problems in the theory of psedoconvex manifolds (cf. \cite{lutian},  \cite{LoiArezzo}). 
Furthermore, in  the noncompact case,   one can find  in \cite{taubnut} a characterization of the flat metric as a Taub-NUT metric with $a_3=0$, while Z. Feng and Z. Tu \cite{fengtu} solve a conjecture formulated in \cite{zedda} by showing that  the complex hyperbolic space is the only Cartan-Hartogs domain where the coefficient  $a_2$ is constant. In \cite{LZhs} A. Loi and M. Zedda prove that a locally hermitian symmetric space with vanishing $a_1$ and $a_2$ is flat. 

The present paper deals with {\it radial}  K\"ahler metrics, namely those \K\ metrics
admitting  a  K\"ahler potential  which depends only on the sum $|z|^2 = |z_1|^2 + \cdots + |z_n|^2$ of the moduli of a local coordinates' system $z_i$. 

Our main results are the following
 
\begin{theor}\label{mainteor}
The third  \tyz coefficient $a_3$ of a radial \K metric with constant scalar curvature (cscK metrc) is constant if and only if the second  \tyz coefficient $a_2$ is constant.  
\end{theor}
\begin{corol}\label{mainteor2}
The flat metric and the Simanca metric are the only radial projectively induced metrics with $a_1 = a_3 = 0$.
\end{corol}

The proof of Theorem \ref{mainteor} is  based on Feng's work \cite{feng} which classifies the radial metrics whose first and second coeffcients of TYZ expansion are constant functions.
Since the proof of  Feng's classification theorem given in \cite{feng} is quite technical and not  so easy to read we have provided in the next section a more readable  
(to the authors' opinion) proof. In the last section we prove Theorem \ref{mainteor} and Corollary \ref{mainteor2}.

\section{Radial metrics with constant $a_1$ e $a_2$}
\begin{theor}[Z. Feng \cite{feng}]\label{radiala2}
The only radial \K potentials defined on a complex domain of dimension $n$ which have constant first and second TYZ coefficient  are:
\begin{itemize}
\item[1.]  the Euclidean metric,\\
\item[2.] constant multiple of the hyperbolic metric defined on $\B^n_{\frac{1}{\lambda}}$,\\
\item[3.] constant multiple of the Fubini-Study metric.
\end{itemize}
Moreover, if $n=2$
\begin{itemize}
\item[4.] $|z|^2+\lambda\log |z|^2$ on $\C^2\setminus\{ 0\}$, \\
\item[5.] $\mu|z|^2-\lambda\log |z|^2$ on $\C^2\setminus\overline{\B}^2_{\frac{\lambda}{\mu}} $,\\
\item[6.] $-\mu \big(\log(1-\xi|z|^{2\zeta})+\frac{1-\zeta}{2}\log |z|^2\big)$  on 
${\B^2}_{(\frac{1}{\xi})^{\frac{1}{\zeta}}} \setminus\overline{\B}^2_{((\frac{1}{\xi})(\frac{1-\zeta}{1+\zeta}))^{\frac{1}{\zeta}}}$,\\
\item[7.] $-\mu \big(\log(1-\xi|z|^{2(\lambda+1)})-\frac{\lambda}{2}\log |z|^2\big)$  on 
 $\B^2_{(\frac{1}{\xi})^{\frac{1}{\lambda+1}}} \setminus\{ 0\}$,\\
\item[8.] $\mu \big(\log(1+\xi|z|^{-2(\lambda+1)})+\frac{2+\lambda}{2}\log |z|^2\big)$  on 
$\C^2\setminus{\B}^2_{((\frac{1}{\xi})(\frac{\lambda}{2+\lambda}))^{\frac{1}{\lambda+1}}}$,\\
\item[9.] $\mu \big(\log(1+|z|^{-2\zeta})+\frac{1+\zeta}{2}\log |z|^2\big)$  on $\C^2\setminus\{ 0\}$,\\
\item[10a.] $-\mu\big(\log (-\log |z|^2+\kappa)+\frac{\log |z|^2}{2}\big)$  on $\B^2_{e^{\kappa}}\setminus\overline{\B}^2_{e^{\kappa-2}} $,\\
\item[11a.] $-\mu\big(\log | \cos(\lambda\log |z|^2+\kappa)|+\frac{\log |z|^2}{2}\big)$  on $\B^2_{r_1(h,\kappa,\lambda)}\setminus\overline{\B}^2_{r_3(h,\kappa,\lambda)} $,\\
\end{itemize}
or if $n=1$
\begin{itemize}
\item[10b.] $-\mu\log |-\log |z|^2+\kappa|$  on $\C \setminus(\de \B_{e^{\kappa}}\cup\{ 0\})$,\\
\item[11b.] $-\mu\log |\cos(\lambda\log |z|^2+\kappa)|$, on $\B_{r_1(h,\kappa,\lambda)}\setminus\overline{\B}_{r_2(h,\kappa,\lambda)} $.\\
\end{itemize}
Where of $\B^n_r$ denotes the ball of radius $r$ in $\C^n$, $\de\B^n_r$ denotes the boundary of $\B^n_r$, $\overline{\B}^n_r=\B^n_r\cup\de\B^n_r$, $\mu,\lambda,\xi,\zeta\in\R^+,$ $0<\zeta<1,$ $\kappa\in\R,$ $h\in\Z,$ $r_1(h,\kappa)=$ $e^{\frac{1}{\lambda}(\frac{2h+1}{2}\pi-\kappa)},$
$r_2(h,\kappa,\lambda)=e^{\frac{1}{\lambda}(\frac{2h-1}{2}\pi-\kappa)}$ and
 $r_3(h,\kappa,\lambda)=e^{\frac{1}{\lambda}(h\pi+\arctan(\frac{1}{2\lambda})-\kappa)}.$  
\end{theor}
\begin{proof}
Let $\Phi (\log r)$ be a \K potential defined on a radial complex domain of complex dimension $n$, where $r=|z_1|^2+\mathellipsis +|z_n|^2$.\\

Firstly, we classify \K metrics related to such kind of potentials which have constant scalar curvature (namely $a_1$ is constant). By definition, the metric tensor reads as
\begin{equation}\label{metric}
g_{i\bar j}=\frac{\de^2 \Phi (\log r)}{\de z_i\de\bar z_j}=\frac{\Phi''-\Phi' }{r^2}\bar z_i z_j+\frac{\Phi' }{r}\delta_{ij},
\end{equation}
where $\delta_{ij}$ is the Kronecker delta and $\Phi'$ represents the first derivative of $\Phi$ with respect to $$t=\log r.$$ Since Riemannian metrics are positive definite, $\Phi''$ needs to be a positive function. Hence we can consider the following substitutions
$$\begin{cases}
 y= \Phi'(t),\\
\psi(y)= \Phi''(t).
\end{cases}$$
From eq. (\ref{metric}), we easily get 
$$\det \left(g_{i\bar j}\right)=\frac{(\Phi')^{n-1}\Phi''}{r^{n}}.$$
By considering that
$$\frac{d}{dt}\big( (n-1)\log\Phi'+\log\Phi''\big)=(n-1)\frac{\Phi''}{\Phi'}+\frac{\Phi'''}{\Phi''}=(n-1)\frac{\psi}{y}+\psi',$$
we introduce the further substitution
$$\sigma(y)=\frac{(y^{n-1}\psi)'}{y^{n-1}}=(n-1)\frac{\psi(y)}{y}+\psi'(y),$$
and we compute the Ricci tensor's components:
\begin{equation}\label{ricci}
\Ric_{i\bar j}=-\frac{\de^2 \log\det \left(g_{i\bar j}\right)}{\de z_i\de\bar z_j}=\frac{-\sigma'\psi+\sigma-n }{r^2}\bar z_i z_j+\frac{n-\sigma }{r}\delta_{ij},
\end{equation}
to be in the position to represent the scalar curvature as a function of $y$:
$$\rho=g^{i\bar j}Ric_{i \bar j}= \frac{n(n-1)}{y}-\frac{(n-1)\sigma+\sigma'y}{y}.$$
Thus we reach our initial objective by solving the ODE obtained by imposing $\rho$ to be constant, namely
$$\frac{n(n-1)}{y}-\frac{(y^{n-1}\psi)''}{y^{n-1}}=-An(n+1),$$
whose solutions are
\begin{equation}\label{solution}
\psi(y)=Ay^2+y+\frac{B}{y^{n-2}}+\frac{C}{y^{n-1}},
\end{equation}
where $A,B,C \in \R$.\\

Now, we determinate which conditions have to be satisfied so that the previous metrics also verify the PDE 
\begin{equation}\label{eq}
a_2=K,
\end{equation}
where $K$ is a real constant.

We are going to represent the first term in the previous PDE (see eq. \ref{coefflu}) as a function of $y$, $\psi(y)$ and its derivatives, in order to convert it to an ODE.

Riemann tensor's components $R_{i\bar j k \bar l}$ are by definition equal to $\frac{\partial^2 g_{i\bar l}}{\partial z_k\partial\bar z_j}-g^{p\bar q}
\frac{\partial g_{i\bar p}}{\partial z_k}\frac{\partial g_{q\bar l}}{\partial\bar z_j}$. Therefore, the following derivatives (cf. eq. \ref{metric})
\begin{equation}\begin{array}{l}\label{gderivatives}
\frac{\partial g_{i\bar l}}{\partial z_k}= \frac{\Phi''-\Phi' }{r^2} (\delta_{kl}\bar z_i +\delta_{il}\bar z_k)+\frac{\Phi'''-3\Phi''+2\Phi' }{r^3} \bar z_i z_l\bar z_k\\
\frac{\partial^2 g_{i\bar l}}{\partial z_k\partial\bar z_j}=
\frac{\Phi''-\Phi' }{r^2} (\delta_{kl}\delta_{ij} +\delta_{il}\delta_{kj})+\frac{\Phi''''-6\Phi'''+11\Phi''-6\Phi' }{r^4} \bar z_i z_jz_l\bar z_k \\
\qquad \qquad +\frac{\Phi'''-3\Phi''+2\Phi' }{r^3} (\delta_{kj}\bar z_i z_l+\delta_{kl}z_j\bar z_i +\delta_{il}z_j\bar z_k+\delta_{ij} z_l\bar z_k)
\end{array}\end{equation}
allow us to state that the unique (up to consider tensor's symmetries) nonvanishing  components in $(z_1,0,\mathellipsis ,0)$ are equal to
\begin{equation}\begin{array}{l}\label{riemann}
R_{1\bar 1 1 \bar 1}=\frac{\psi''\psi^2}{r^2},\\
R_{1\bar 1 i \bar i}=\frac{\psi'y-\psi}{yr^2}\psi,\\
R_{i\bar i i \bar i}=2 R_{i\bar i j \bar j}=2\frac{\psi-y}{r^2},\\
\end{array}\end{equation}
where $i \neq j$ and $i,j\neq 1$. By taking into account the invariance of $|R|^2$ under the action of the unitary group, we can use Riemann tensor's components of formula (\ref{riemann}) and the metric tensor's components evaluated in $(z_1,0,\mathellipsis ,0)$ (cf. eq. \ref{metric}) to get the general formula
\begin{equation}\begin{array}{ll}\label{riemannmod}
|R|^2 &
=(g^{ 1\bar 1})^4 (R_{1\bar 1 1 \bar 1})^2 + 4(n-1)(g^{1 \bar 1})^2(g^{ i\bar i })^2(R_{1\bar 1 i \bar i})^2  \\  
& \qquad + (g^{i \bar i})^4\big( 4 \frac{(n-2)(n-1)}{2}(R_{i\bar i j \bar j})^2+ (n-1)(R_{i\bar i i \bar i})^2\big)  \\ 
& =(\psi'')^2+4(n-1)\big(\frac{\psi'y-\psi}{y^2}\big)^2+2n(n-1)\big(\frac{\psi-y}{y^2}\big)^2.
\end{array}\end{equation}
Similarly, we compute
\begin{equation}\begin{array}{ll}\label{riccimod}
|Ric|^2 &=
(g^{ 1\bar 1})^2 (Ric_{1\bar 1})^2+(n-1)(g^{ i\bar i})^2 (Ric_{i\bar i})^2\\  &=
(\sigma')^2+(n-1)\big(\frac{\sigma-n}{y} \big)^2.
\end{array}\end{equation}
Therefore the equation (\ref{eq}) is equivalent to
\begin{multline}
4(n-1)\Big(\frac{\psi'y-\psi}{y^2}\Big)^2+2n(n-1)\Big(\frac{\psi-y}{y^2}\Big)^2\\
-4\Big((\sigma')^2+(n-1)\Big(\frac{\sigma-n}{y}\big)\Big)
+(\psi'')^2  \\=  24K-3A^2n^2(n+1)^2.\nonumber
\end{multline}
By imposing in the previous ODE, $\psi$ to be equal to the solution (\ref{solution}), we get
\begin{equation}\begin{array}{l}\label{ode}
(n-1)\ \Big(
2BC n^2y^{- 2n - 1}(n + 1) + C^2ny^{- 2(n + 1)}(n + 1)(n + 2) \\
\ +B^2ny^{- 2 n}(n + 1)(n - 2) + A^2n(n + 1)(3n + 2)\Big)
-24K=0.\nonumber
\end{array}\end{equation}
Hence,  the previous equation is satisfied  if and only if  
\begin{equation}
\begin{cases}
C=B=0, \ K=A^2n(n+1)(n-1)(3n+2)/24 & \text{if }n\geq 2\\
C=0, \ K=2A^2&  \text{if }n= 2\\
K=0 &  \text{if }n= 1\\
\end{cases}\nonumber
\end{equation}
By definition of $\psi$ and $y$ and by renaming the constants for more convenience, one gets
that $a_2=K$ if and only if 
\begin{equation}\label{psi''}
\Phi''=
\begin{cases}
a(\Phi')^2 + \Phi' & \text{if }n\geq 2,\\
a(\Phi')^2+\Phi'+c &  \text{if }n= 2,\\
a(\Phi')^2+b\Phi'+c &  \text{if }n= 1.
\end{cases}
\end{equation}\\

In order to solve this equation
let us first  suppose $a\neq 0$. Therefore
\begin{equation}\label{initialode}
\Phi''(t)=a\left(\left(\Phi'+\frac{b}{2a}\right)^2+\frac{c}{a}-\frac{b^2}{4a^2}\right).
\end{equation}
We distinguish now three different cases.

\noindent
Case 1: $\frac{c}{a}-\frac{b^2}{4a^2}=-D^2$, where $D\in\R^+$. By solving this ODE, we get $$\left|\frac{\Phi'(t)+\frac{b}{2a}-D}{\Phi'(t)+\frac{b}{2a}+D}\right|=\xi e^{2aDt},$$
 where $\xi\in\R^+$. Thus we have to distinguish two further different possibilities. 
 
1a) $\Phi'+\frac{b}{2a}-D<0<\Phi'+\frac{b}{2a}+D$.
 
Hence $$\Phi'(t)=-\frac{\xi e^{2aDt}}{\xi e^{2aDt}+1}\left(D+\frac{b}{2a}\right)+\frac{D-\frac{b}{2a}}{\xi e^{2aDt}+1},$$ namely 
 $$\Phi(\log r)=-\frac{1}{a}\left(\log(1+\xi r^{2aD})+\left(\frac{b}{2}-Da\right)\log r\right). $$
  Since $\Phi''>0$, then $a<0$. If $n\geq 2$, also $\Phi'$ has to be positive by definition of \K potential.  This condition is trivially satisfied if $n>2$, while if $n=2$ it is equivalent to $(2aD+1)\xi r^{2aD}>2aD-1$. Moreover, if $n=1$, we recall that $\de\bar\de\log r=0$.\\

1b) $\Phi'+\frac{b}{2a}-D>0$ or $\Phi'+\frac{b}{2a}+D<0$. Therefore $$\Phi'(t)=-\frac{\xi e^{2aDt}}{\xi e^{2aDt}-1}\left(D+\frac{b}{2a}\right)-\frac{D-\frac{b}{2a}}{\xi e^{2aDt}-1},$$ namely
$$\Phi(\log r)=-\frac{1}{a}\left(\log|1-\xi r^{2aD}|+\left(\frac{b}{2}-Da\right)\log r\right). $$
 Since $\Phi''>0$, then $a>0$. If $n\geq 2$, also $\Phi'$ has to be positive.  This condition is satisfied if and only if $ \xi r^{2aD}<1$ and   $\xi r^{2aD}>\frac{1-2aD}{2aD+1}$ if $n=2$.

\noindent
Case 2: $\frac{c}{a}-\frac{b^2}{4a^2}=0$. This case may occur only if $n=1$ or $n=2$. By solving the initial ODE (\ref{initialode}), we get    $$\Phi(\log r)=-\frac{1}{a}\log|\log r+\kappa|-\frac{b}{2a}\log r,$$
where $\kappa\in\R$. Since $\Phi''>0$, then $a>0$. If $n=2$, also $\Phi'$ has to be positive, hence $-2-\kappa<\log r<-\kappa$ (in this case $b=1$). 

\noindent
Case 3: $\frac{c}{a}-\frac{b^2}{4a^2}=D^2$. Also this case may occur only if $n=1$ or $n=2$. By solving the initial ODE (\ref{initialode}), we get   
$$\Phi(\log r)= -\frac{1}{a}\log|\cos(aD\log r+\kappa)|-\frac{b}{2a}\log r,$$
 where $\kappa\in\R$. Since $\Phi''>0$, then $a>0$. If $n=2$, also $\Phi'$ has to be positive, hence $\frac{1}{aD}(\arctan(\frac{1}{2aD})+ h\pi -\kappa)<\log r<\frac{1}{aD}(\frac{\pi}{2}+h\pi-\kappa)$, where $h\in\Z$.\\

To conclude we have to consider that $a$ may also be equal to $0$, hence we have to solve  $$\Phi''(t)=b\Phi'+c.$$

If $b\neq 0$,  we get 
$$\Phi(\log r)=\frac{1}{b^2}\xi r^{b}- \frac{c}{b} \log r.$$
If $n\geq2$,  $\Phi'$ has to be positive, hence $\xi r^{b}>\frac{c}{b}$. This condition is always satisfied if $n>2$, since $c=0$. Moreover $\Phi''$ is always positive.

Instead, if $b=0$ (this case may occur only if $n=1$), we get $$\Phi(\log r)=\frac{c}{2} \left(\log r\right)^2+ \kappa \log r,$$
which is the Euclidean metric (already considered before).
\end{proof}

\section{Proof of Theorem \ref{mainteor} and Corollary \ref{mainteor2}}
In order to prove Theorem \ref{mainteor} we recall the expression of the term $a_3$ in \tyz expansion:
\begin{equation}\label{coeffa3}
\begin{array}
{l}
a_3(x)=\frac{1}{8}\Delta\Delta\rho +\frac{1}{24}{\rm div}{\rm div} (R, {\rm Ric})- \frac{1}{6}{\rm div}{\rm div} (\rho{\rm Ric})\\
\qquad \quad  +\frac{1}{48}\Delta (|R|^2-4|{\rm Ric} |^2+8\rho ^2)+
\frac{1}{48}\rho(\rho ^2- 4|{\rm Ric} |^2+ |R|^2)\\
\qquad \quad +\frac{1}{24}(\sigma_3 ({\rm Ric})- {\rm Ric} (R, R)-R({\rm Ric} ,{\rm Ric})),
\end{array}
\end{equation}
where we are using  the following notations (in local coordinates $z_1, \dots , z_n$):
$$\begin{array}{l}\label{values}
|D^{'}\rho|^2=  g^{j \bar i} \frac{\partial \rho}{\partial z_i} \frac{\partial \rho}{\partial \bar z_j},\\
|D^{'}\Ric|^2= g^{\alpha \bar i} g^{j \bar \beta} g^{\gamma \bar k}  \Ric_{i\bar j, k} \overline{\Ric_{\alpha \bar \beta, \gamma}} ,\\
|D^{'}R|^2 = g^{\alpha \bar i}g^{j \bar \beta}g^{\gamma \bar k}g^{l \bar \delta}g^{\epsilon \bar p} R_{i \bar j k \bar l, p} \overline{R_{\alpha \bar \beta \gamma \bar \delta, \epsilon}} ,\\
\div\div (\rho \Ric)=2|D^{'}\rho|^2+
g^{\beta \bar i} g^{j \bar \alpha} \Ric_{i\bar j}\frac{\partial^2 \rho}{\partial z_{\alpha} \partial \bar z_{\beta}}
+\rho\Delta\rho,\\
\div\div (R, \Ric)=
-g^{\beta \bar i} g^{j \bar \alpha} \Ric_{i\bar j}\frac{\partial^2 \rho}{\partial z_{\alpha} \partial \bar z_{\beta}}
-2|D^{'}\Ric|^2\\
\qquad\qquad\qquad + g^{\alpha \bar i} g^{j \bar \beta} g^{\gamma \bar k} g^{l \bar \delta} R_{i\bar j,k\bar l}R_{\beta \bar \alpha \delta \bar \gamma}-
R(\Ric, \Ric)-\sigma_3(\Ric),\\
R(\Ric, \Ric)= g^{\alpha \bar i}g^{j \bar \beta}g^{\gamma \bar k}g^{l \bar \delta} R_{i\bar jk\bar l}\Ric_{\beta \bar \alpha}\Ric_{\delta \bar \gamma},\\
\Ric(R, R) = g^{\alpha \bar i}g^{j \bar \beta}g^{\gamma \bar k}g^{\delta \bar p}g^{q \bar \epsilon} \Ric_{i\bar j}R_{\beta \bar \gamma p \bar q}R_{k \bar \alpha \epsilon \bar \delta},\\
\sigma_3 (\Ric)=g^{\delta \bar i}g^{j \bar \alpha}g^{\beta \bar \gamma} \Ric_{i\bar j}\Ric_{\alpha \bar \beta}\Ric_{\gamma \bar \delta},\\
\end{array}$$
where the $g^{j \bar i}$'s denote the entries of the inverse matrix of the metric (i.e. $g_{k \bar i} g^{j \bar i} = \delta_{kj}$) and  \lq\lq\  ,p'' represents the covariant derivative in the direction $\frac{\partial}{\partial z_p}$ and we are using the summation convention for repeated indices.
 The reader is also referred to  \cite{loianal} and \cite{loismooth}  for a  recursive formula  for the coefficients $a_j$'s  and  an alternative computation of  $a_j$ for $j\leq 3$ using Calabi's diastasis function (see also  \cite{xu1}  for a graph-theoretic interpretation of this recursive formula). Moreover note that, given any K\"ahler manifold $(M, g)$ it makes sense to call the $a_j$'s Tian-Yau-Zelditch coefficients, regardless of the existence of  \tyz expansion.

\begin{proof}[Proof of Theorem \ref{mainteor}]

The same arguments about unitary invariance of terms of $a_2$ shown in the proof of Theor. \ref{radiala2} hold true also in this case for the term $a_3$. Therefore, we can easily compute by using  eq. (\ref{riemann}) and (\ref{metric}) evaluated at $(z_1,0,\mathellipsis ,0)$ the following terms:
\begin{equation}\begin{array}{l}\label{ricr}
{\rm Ric} (R, R)\\
\quad=(g^{ 1\bar 1})^2 Ric_{1\bar 1}\big( (g^{ 1\bar 1})^3 (R_{1\bar 1 1 \bar 1})^2 + 2 (n-1)g^{ 1\bar 1}(g^{ i\bar i})^2 (R_{1\bar 1 i \bar i})^2 \big) \\
\quad\quad\quad  + (g^{ i\bar i})^5 Ric_{i\bar i} \big( (n-1) (R_{i\bar i i \bar i})^2 + 2(n-1)(n-2)(R_{i\bar i j \bar j})^2 \big)\\
\quad\quad\quad   +2 (n-1)(g^{ 1\bar 1})^2(g^{ i\bar i})^3 Ric_{i\bar i}(R_{1\bar 1 i \bar i})^2  \\
 \quad =2(n-1)\frac{n-\sigma}{y^5}\Big(n(\psi-y)^2+(\psi'y-\psi)^2\Big)\\
\quad\quad\quad  -\sigma'\Big( (\psi'' )^2+2(n-1)\frac{(\psi'y-\psi )^2}{y^4}\Big);
\end{array}\end{equation}
\begin{equation}\begin{array}{l}\label{rric}
R({\rm Ric} ,{\rm Ric})\\
\quad =(g^{ 1\bar 1})^4 (Ric_{1\bar 1})^2 R_{1\bar 1 1 \bar 1}+2(n-1)(g^{ 1\bar 1})^2 (g^{ i\bar i})^2 Ric_{1\bar 1}Ric_{i\bar i} R_{1\bar 1 i \bar i}\\
\quad \quad\quad  + (g^{ i\bar i})^4 (Ric_{i\bar i})^2 ( (n-1)R_{i\bar i i \bar i}+(n-1)(n-2)R_{i\bar i j \bar j})\\
\quad =(\sigma')^2\psi''-2(n-1)\frac{(\psi'y-\psi)(n-\sigma)\sigma'}{y^3}+n(n-1)\frac{(\psi-y)(n-\sigma)^2}{y^4}.
\end{array}\end{equation}

By using formulas (\ref{riemannmod}) and (\ref{riccimod}), we define
$$\chi (y)= |R|^2-4|{\rm Ric} |^2,$$
to get 
\begin{multline}\label{lapl}
\Delta (|R|^2-4|{\rm Ric} |^2)=g^{i \bar j}\frac{\de^2 \chi}{\de z_j\de\bar z_i}\\
=g^{i \bar j}\big( \frac{\chi'\psi}{r}\delta_{ji}+\frac{\chi''\psi+\chi'\psi'-\chi'}{r^2}\psi\bar z_j z_i \big)=
(\chi'\psi)'+\frac{n-1}{y}\chi'\psi.
\end{multline}

Since Christoffel's symbols $\Gamma_{ki}^p$ are equal to $g^{p\bar q} \frac{\partial g_{i\bar q}}{\partial z_k}$, we easily deduce from (\ref{gderivatives}) that they are all equal to zero when evaluated at $(z_1,0,\mathellipsis ,0)$ except for:
\begin{equation}\begin{array}{l}\label{christoffel}
\Gamma_{11}^1= (\frac{\Phi'''-\Phi''}{r^2})(\frac{r}{\Phi''})\bar z_1=\frac{\psi'-1}{r}\bar z_1,\\
\Gamma_{1i}^i= (\frac{\Phi''-\Phi'}{r^2})(\frac{r}{\Phi'})\bar z_1=\frac{\psi-y}{yr}\bar z_1.\nonumber
\end{array}\end{equation}
We also compute that the unique first derivatives of Ricci tensor's components different from zero in $(z_1,0,\mathellipsis ,0)$ (cf. (\ref{ricci})) are:
\begin{equation}\begin{array}{l}\label{riccider1}
\frac{\de }{\de\bar z_1}Ric_{1\bar 1}=\frac{\sigma'\psi-(\sigma'\psi)'\psi}{r^2}z_1,\\
\frac{\de }{\de\bar z_1}Ric_{i\bar i}=\frac{\de }{\de\bar z_i}Ric_{i\bar 1}=\frac{\sigma-\sigma'\psi-n}{r^2}z_1,\nonumber
\end{array}\end{equation}
and we use them to evaluate Ricci tensor's first covariant derivatives at $(z_1,0,\mathellipsis ,0)$, which are defined as
$Ric_{i\bar j,k}=\frac{\de}{\de z_k} Ric_{i\bar j}- Ric_{p\bar j}\Gamma_{ki}^p$, and we get
\begin{equation}\begin{array}{ll}\label{ric1}
|D^{'}\Ric|^2 & = (g^{1\bar 1})^3|Ric_{1\bar 1,1}|^2+2(n-1)g^{1\bar 1}(g^{i\bar i})^2|Ric_{i\bar i,1}|^2\\
& =(\sigma'')^2\psi+2(n-1)\frac{\psi}{y^2}(\sigma'+\frac{n-\sigma}{y})^2.
\end{array}\end{equation}

Because some Riemann tensor's components in $(z_1,0,\mathellipsis ,0)$ are equal to zero (cf. eq. (\ref{riemann})), we need to compute just some Ricci tensor's second covariant derivatives, of which we recall the definition: 
$Ric_{i\bar j,k\bar l}=\partial_{\bar l}\partial_k Ric_{i\bar j}+\Gamma_{ki}^q\Gamma_{\bar l\bar j}^{\bar p}Ric_{q\bar p}- \Gamma_{ki}^p\partial_{\bar l}Ric_{p\bar j} -\partial_{\bar l}\Gamma_{ki}^p Ric_{p\bar j} -\Gamma_{\bar l\bar j}^{\bar p}\partial_kRic_{i\bar p}$.
 In particular, we preliminary  evaluate at $(z_1,0,\mathellipsis ,0)$ the following derivatives ($\partial_{\bar l}\Gamma_{ki}^p=g^{p\bar q} \frac{\partial^2 g_{i\bar q}}{\partial \bar z_l\partial z_k} - g^{p\bar \gamma}g^{\alpha\bar q} \frac{\partial g_{\alpha\bar \gamma}}{\partial \bar z_l} \frac{\partial g_{i\bar q}}{\partial z_k}$)
\begin{equation}\begin{array}{l}\label{chrder}
\partial_{\bar 1}\Gamma_{11}^1=\frac{\psi''\psi}{r},\\
\partial_{\bar i}\Gamma_{i1}^1=\frac{\psi'y-\psi}{y r},\\
\partial_{\bar 1}\Gamma_{i1}^i=\frac{\psi'y-\psi}{y^2r}\psi,\\
\partial_{\bar i}\Gamma_{ii}^i=2\partial_{\bar j}\Gamma_{ji}^i=2\frac{\psi-y}{yr}\nonumber
\end{array}\end{equation}
and
\begin{equation}\begin{array}{l}\label{riccider2}
\frac{\de^2 }{\de z_1\de\bar z_1}Ric_{1\bar 1}=\frac{-((\sigma'\psi)'\psi)'\psi+2(\sigma'\psi)'\psi-\sigma'\psi}{r^2},\\
\frac{\de^2 }{\de z_1\de\bar z_1}Ric_{i\bar i}=\frac{\de^2 }{\de z_i\de\bar z_1}Ric_{1\bar i}=\frac{-(\sigma'\psi)'\psi+2\sigma'\psi-\sigma+n}{r^2},\\
\frac{\de^2 }{\de z_i\de\bar z_i}Ric_{i\bar i}=2\frac{\de^2 }{\de z_j\de\bar z_j}Ric_{i\bar i}=2\frac{\de^2 }{\de z_j\de\bar z_i}Ric_{i\bar j}=2\frac{-\sigma'\psi+\sigma-n}{r^2},\nonumber
\end{array}\end{equation}
 to  be in the position to express 
\begin{equation}\begin{array}{l}\label{ric2}
(g^{1 \bar 1})^4 R_{1\bar 1 1\bar 1} Ric_{1\bar 1, 1\bar 1} + 
2 (n-1)(g^{1 \bar 1})^2 (g^{i \bar i})^2 R_{1\bar 1 i\bar i} (Ric_{1\bar 1, i\bar i} +Ric_{i\bar i, 1\bar 1})\\
\qquad\quad  +(g^{i \bar i})^4 \big( (n-1) R_{i\bar i i\bar i} Ric_{i\bar i, i\bar i} +2(n-1)(n-2) R_{i\bar i j\bar j} Ric_{i\bar i, j\bar j}\big)\\
 =\frac{\psi''}{\psi}\Big(-((\sigma'\psi)'\psi)'-(\psi')^2\sigma'+2\psi'(\sigma'\psi)'+\sigma'\psi''\psi\Big)\\
 \qquad\quad +2(n-1)\frac{\psi'y-\psi}{ y^3}\left(-2(\sigma'\psi)'+2\frac{\psi}{y^2}(n-\sigma)+ 4\frac{\psi}{y}\sigma'\right)\\ 
\qquad\quad +2(n-1)\frac{\psi'y-\psi}{ y^3} \left( \psi'-\frac{\psi}{y} \right)\left(\sigma'+\frac{\sigma-n}{y}\right) \\
\qquad\quad +2n\frac{\psi}{y^5}(n-1)(\psi-y)(\sigma-n-\sigma'y)
\end{array}\end{equation}
as a function of $y$, $\psi$ and its derivatives.
 
 Therefore, thanks to the formulas (\ref{riemannmod}), (\ref{riccimod}), (\ref{rric}), (\ref{ricr}), (\ref{lapl}), (\ref{ric1}) and (\ref{ric2}), using (\ref{solution}), we convert the PDE 
\begin{equation}\label{eq2}
a_3=K,
\end{equation}
where $K$ is a real constant, to the identity:
\\
\noindent 
$n(n-1)(n+1)\Big(
A^3ny^{3n + 3}(n + 1)(n - 2) +AB^2ny^{n+3} (2-n) (n+3) -2Bny^{n + 2}\big(AC(n + 1)(n + 4) + B(n^2 - 4)\big) -Cy^{n + 1}(n+1) (n+2) \big(AC(n+6) +4Bn\big) -2C^2y^n(n + 1)(n + 2)(n + 3) +2B^3ny^3 (2-n) (2n+1) -6B^2Cn^2y^2(2n + 1) - 6BC^2y(n + 1)(2n^2 + 3n + 2) - 2C^3(n + 1)(n + 2)(2n + 3)
\Big)=-48 Ky^{3n+3}.
$\\

Therefore, if $n=1$ and $K=0$ the previous equation holds true independently from $A$, $B$ and $C$.
If $n\neq 1$, $C$ needs to be equal to $0$. By putting $C=0$ the  equality
is satisfied if $K=0$ and $n=2$ or $n>2$, $B=0$ and $K=-\frac{A^3}{48} n^2(n+1)^2(n-1)(n-2)$.

To sum up, we have just proved that $a_3=K$  if and only if  
$$\psi(y)=
\begin{cases}
Ay^2 + y & \text{if }n\geq 2,\\
Ay^2+ y+B &  \text{if }n= 2,\\
Ay^2+ By+C &  \text{if }n= 1.
\end{cases}$$
By recalling that $y=\Phi'$ and eq. \eqref{psi''}, we get that $a_2$ is constant if and only if $a_3=K$ is constant, which concludes the proof of Theorem \ref{mainteor}.
\end{proof}

In order to prove Corollary \ref{mainteor2} we need the following result, proved by the authors in \cite{LSZ} (Lemma 2.2 and Corollary 2.3).

\begin{lem}\label{projind}
Let $n\geq 2$ and  $p=(s,0,\mathellipsis ,0)$, with $s\in\R$, $s\neq 0$, be a point of the complex domain $U\subset\C^n\setminus\{0\}$ on which is defined a radial metric $g$  with radial  \K potential $\Phi:U\rightarrow \R$ and corresponding diastasis  $D_p:U\rightarrow \R$.
Let $f:\tilde U \rightarrow \R$ defined by $f (r)=\Phi (z),\  z=(z_1, \dots , z_n)$ where $\tilde U=\{r=|z|^2=|z_1|^2+\cdots +|z_n|^2 \ | \ z\in U\}$ and, for  $h\in\N$,  let $g_h: \tilde U\rightarrow \R$ given by:
\begin{equation}\label{gh}
g_h(r)=\frac{d^{h}e^{f(r)}}{dr^{h}}e^{-f(r)}.
\end{equation}
If there exists  $r \in\tilde U$ and $h\in\N$ such that  the function given by (\ref{gh}) is negative, namely $g_h(r)<0$, then 
the metric $g$ is not projectively induced.
\end{lem}

\begin{proof}[Proof of Corollary \ref{mainteor2}] Let us first recall that the {\it Simanca metric} is the K\"ahler metric $g_S$ on $\C^2 \setminus \{ 0 \}$ defined by the K\"ahler potential

\begin{equation}\label{simancapiu2}
|z|^2+ \log |z|^2
\end{equation}

This metric, used by Simanca in \cite{simanca} to construct constant scalar curvature metrics on blowups, was studied by the authors in \cite{LSZ} (Theorem 1.3), where it was proved that it is projectively induced.

Now, Theorem \ref{mainteor} implies that that the Simanca metric, the flat metric on $\C^n$ and the metric on $\{ z \in \C^2 \ | \ |z|^2 > 1 \}$ given by the potential

\begin{equation}\label{simancameno2}
|z|^2- \log |z|^2
\end{equation}
are, up to homothety, the only radial  K\"ahler metrics $g$ with $a_1 = a_3= 0$.
Indeed, more precisely, by Theorem \ref{mainteor}, if $a_1 = a_3= 0$ then $g$ must be one of the metrics of the list in the statement of Theorem \ref{radiala2}. In particular, since $a_1 = 0$ we are in the case $\Phi'' = b \Phi' + c$ (being $\Phi$ a Kahler potential of $g$) considered in the proof at the end of Section 2: after an easy integration of this ODE, as seen above, one gets the Euclidean metric and the metrics given by $4.$ and $5.$ of the list in the statement of Theorem \ref{radiala2}, which, respectively after the change of variables $w = \frac{z}{\sqrt{\lambda}}$ and $w = \sqrt{\frac{\mu}{\lambda}} z$, proves our claim.

Since it is well-known and easily verified that  the Euclidean metric  is projectively induced (see e.g. \cite{Cal}),  in order to conclude the proof of Theorem \ref{mainteor2} we need  to show that the \K metrics whose \K forms  are given by  $\omega=\lambda\frac{i}{2}\partial\bar\partial (|z|^2- \log |z|^2)$ are not projectively induced, for any $\lambda>0$.

Set  $f(r) = \lambda(r - \log r)$, $r=|z|^2$.
Thus $e^{f(r)} = e^{\lambda r} r^{- \lambda}$ and $g_3(r)$ in the previous Lemma reads as: 
\begin{equation}\label{g3}
g_3(r) = \frac{d e^{f}}{dr^3}e^{-f} =\lambda r^{ - 3}[\lambda^2 r^3 - 3 \lambda^2 r^2 + 3 \lambda (\lambda + 1)r -(\lambda+1)(\lambda+2)].
\end{equation}

Since the metrics are defined on $\{ z \in \C^2 \ | \ |z|^2 > 1 \}$, we can let $r$ tend to $1$ in (\ref{g3}) and get
$$\lim_{r \rightarrow 1^+} g_3(r) = \lambda[\lambda^2 - 3 \lambda^2  + 3 \lambda^2 + 3 \lambda - \lambda^2 - 3 \lambda - 2] = -2 \lambda < 0$$
which, by Lemma \ref{projind} proves our claim and  concludes the proof of Corollary \ref{mainteor2}.
\end{proof}

\end{document}